\documentclass[11pt]{amsart}
\usepackage{amsmath}
\usepackage{amssymb,amsmath,amsthm,mathrsfs}
\usepackage{amssymb}
\usepackage[margin=1in]{geometry}
\usepackage{graphicx}
\usepackage{caption}
\usepackage{subcaption}
\usepackage{amsthm}
\usepackage{setspace}
\usepackage[final]{showkeys}
\usepackage{bbm}

\newtheorem{thm}{Theorem}[section]
\newtheorem{lemma}{Lemma}[section]
\newtheorem{prop}[thm]{Proposition}

\theoremstyle{definition}
\newtheorem{defn}[thm]{Definition}
\newtheorem{remark}{Remark}[thm]
\newtheorem*{condition}{Condition}

\newcommand{\set}[1]{\left\{#1\right\}}
\newcommand{\C}{\mathcal{O}(1)}
\newcommand{\abs}[1]{\left|#1\right|}
\newcommand{\R}{\mathbb{R}}

\newcommand{\floor}[1]{\lfloor#1\rfloor}

\def\R{\ensuremath{\mathbb R}}
\def\N{\ensuremath{\mathbb N}}

\def\I{\ensuremath{{\mathbbm{1}}}}
\def\e{\ensuremath{\text e}}

\def\R{\ensuremath{\mathbb R}}
\def\N{\ensuremath{\mathbb N}}

\def\I{\ensuremath{{\bf 1}}}
\def\e{\ensuremath{\text e}}

\def\S{\ensuremath{\mathcal S}}
\def\RR{\ensuremath{\mathcal R}}

\def\C{\ensuremath{\mathcal C}}

\def\p{\ensuremath{\mathbb P}}

\def\ie{{\em i.e.}, }

\def\E{\mathbb E}

\def\eps{\varepsilon}

\title[Poisson dispersing billiards]
{Poisson return time statistics for  dispersing billiard maps and \?.}
\date{\today}

\begin{document}

\title[]{Convergence of Rare Events Point Processes to the Poisson for billiards}

\author[J. M. Freitas]{Jorge Milhazes Freitas}
\address{Jorge Milhazes Freitas\\ Centro de Matem\'{a}tica \& Faculdade de Ci\^encias da Universidade do Porto\\ Rua do
Campo Alegre 687\\ 4169-007 Porto\\ Portugal}
\email{jmfreita@fc.up.pt}
\urladdr{http://www.fc.up.pt/pessoas/jmfreita}

\author[N. Haydn]{Nicolai Haydn}
\address{Nicolai Haydn\\ Department of Mathematics\\
University of Southern California\\
Los Angeles\\
90089-2532\\
USA} \email{nhaydn@usc.edu}

\author[M. Nicol]{Matthew Nicol}
\address{Matthew Nicol\\ Department of Mathematics\\
University of Houston\\
Houston\\
TX 77204\\
USA} \email{nicol@math.uh.edu}
\urladdr{http://www.math.uh.edu/~nicol/}

\begin{abstract}
We show that for planar dispersing billiards the return times distribution is, in the 
limit, Poisson for metric balls almost everywhere w.r.t.\ the SRB measure. 
Since the Poincar\'e return map is piecewise smooth 
but  becomes singular at the boundaries of the partition elements, recent results on
the limiting distribution of return times cannot be applied as they require 
the maps to have bounded second derivatives everywhere. We first prove the Poisson limiting 
distribution assuming exponentially decaying correlations. 
For the case when the correlations decay polynomially,
 we induce on a subset on which the induced map has exponentially decaying correlations. 
 We then  prove a general theorem according to which the limiting return times statistics of the 
 original map and the induced map are the same.

\end{abstract}

\thanks{JMF was partially supported by FCT grant SFRH/BPD/66040/2009, by FCT (Portugal) projects PTDC/MAT/099493/2008 and PTDC/MAT/120346/2010, which are financed by national and European Community structural funds through the programs  FEDER and COMPETE . JMF was also supported by CMUP, which is financed by FCT (Portugal) through the programs POCTI and POSI, with national
and European Community structural funds.}

\date{\today}

\maketitle

\section{Introduction}

The purpose of this work is to study the statistical laws ruling the occurrence of rare events for  billiards. The starting point is the analysis of stationary stochastic processes $X_0, X_1,\ldots$ generated by the dynamics of the billiards considered. The rare events will be the exceedances of an high threshold $u$, meaning the occurrence of the event $X_j>u$, for some $j\in\N_0$, which will correspond to the entrance of the orbit, at time $j\in\N_0$, in a small region of the phase space, namely in small neighbourhood of a certain point $\zeta$ chosen in the phase space.  

We will consider the Rare Events Point Processes (REPP), which keep record of the number of exceedances (or entrances in such small balls around $\zeta$) in a certain normalised time interval. When the waiting times (conveniently normalised) between the occurrence of rare events is typically exponential, then one expects the REPP to assume a Poisson type behaviour. 

Recently Chazottes and Collet, \cite{CC13} showed that for any two-dimensional dynamical system $(T,X,\mu)$ modeled by a Young Tower which has bounded
derivative and  exponential tails (and hence exponential
decay of correlations for H\"older observations) the REPP converges typically in distribution to a Poisson process, when the balls around $\zeta$ shrink towards its centre. 
They also gave rates of convergence. Unfortunately their proof relies  on both 
the boundedness of the derivative of $T$, $|DT|_{\infty}<C$, and exponential tails. Their results do not apply to 
exponentially mixing Sinai dispersing billiards (which have unbounded derivative) nor to billiard systems with polynomial
rates of decay of correlation. 

Our goal is to show that for planar Sinai dispersing billiards (with finite or infinite horizon) and also for certain billiard systems with polynomial decay of correlations, the REPP, typically, converges in distribution to a standard Poisson process, when the thresholds $u$ converges to the maximum value attainable and the corresponding neighbourhoods shrink to $\zeta$. This typically means that the convergence of the REPP to a standard Poisson occurs for a.e.\ point $\zeta$ chosen in the phase space, with respect to the invariant measure, which, in our case, is equivalent to Lebesgue measure.

Note that there are two perspectives to look at rare events in a dynamical setting: one consists in looking at the exceedances as extreme values for the random variables $X_j$, for $j\in\N_0$, in which case, one uses tools of Extreme Value Theory; the other consists in looking at rare events as hits or returns by the orbits to small sets in the phase space which is tied to the phenomenon of recurrence. These two perspectives are linked and essentially they are just two sides of the same coin. This connection was first observed in \cite{C01} and formally established, in \cite{FFT10,FFT11}.  

Our proofs are based upon extreme value theory and some remarkable ideas of Collet~\cite{C01}. We first give proofs for Sinai dispersing billiards, then show how recent work of Chernov and Zhang~\cite{CZ05} and  Markarian~\cite{M04} allows us to extend these results to billiards with polynomial
decay by inducing on a subset for which the return map has good hyperbolic properties. 

\section{The setting and statement of results}

Let $(T,X,\mu)$ be an ergodic transformation of  a probability space. We suppose that $X$ is 
embedded in a   Riemannian manifold of dimension $d$. Suppose that the time series $X_0, X_1,\ldots$ arises from such a system simply by evaluating a given  observable $\varphi:X\to\R\cup\{\pm\infty\}$ along the orbits of the system, or in other words, the time evolution given by successive iterations by $T$:
\begin{equation}
\label{eq:def-stat-stoch-proc-DS} X_n=\varphi\circ T^n,\quad \mbox{for
each } n\in {\mathbb N}.
\end{equation}
Clearly, $X_0, X_1,\ldots$ defined in this way is not an independent sequence.  However, $T$-invariance of $\mu$ guarantees that this stochastic process is stationary.

We suppose that the r.v. $\varphi:X\to\R\cup\{\pm\infty\}$
achieves a global maximum at $\zeta\in X$ (we allow
$\varphi(\zeta)=+\infty$).  We assume that $\varphi$ and $\mu$ are sufficiently regular so that, for $u$ sufficiently close to $u_F:=\varphi(\zeta)$,  the event 
\begin{equation*}
\label{def:U}
U(u):=\{x\in X:\; \varphi(x)>u\}=\{X_0>u\}
\end{equation*} corresponds to a topological ball centred at $\zeta$. Moreover, the quantity $\mu(U(u))$, as a function of $u$, varies continuously on a neighbourhood of $u_F$.

We are interested in studying the extremal behaviour of the stochastic process $X_0, X_1,\ldots$ which is tied to the occurrence of exceedances of high levels $u$. The occurrence of an exceedance at time $j\in\N_0$ means that the event $\{X_j>u\}$ occurs, where $u$ is close to $u_F$. Observe that a realisation of the stochastic process $X_0, X_1,\ldots$ is achieved if we pick, at random and according to the measure $\mu$, a point $x\in X$, compute its orbit and evaluate $\varphi$ along it. Then saying that an exceedance occurs at time $j$ means that the orbit of the point $x$ hits the ball $U(u)$ at time $j$, \ie $f^j(x)\in U(u)$. 

For more details on the choice of the observables so that the above properties hold and the link between extreme values and hitting/returns to small sets endures we suggest the reader look at \cite[Section~4.1]{F13}. However, for definiteness we mention that a possible choice for $\varphi$ in this setting, where the invariant measure $\mu$ will be equivalent to Lebesgue measure, is the following:
consider some point $\zeta\in X$ and take 
\begin{equation}
\label{eq:observable-particular}
\varphi(x)=-\log (\mbox{dist}(x,\zeta)),
\end{equation}
 where $\mbox{dist}(\cdot,\cdot)$ denotes the usual euclidean metric in $X$.

A very important issue in order to take limits is to establish the rate of convergence of $u$ to $u_F$. For that we will consider sequences $(u_n)_{n\in\N}$ such that
\begin{equation}
\label{eq:def-un}
\lim_{n\to\infty}n\mu(X_0>u_n)=\tau,
\end{equation}
for some $\tau\geq 0$.
The motivation for using such normalising sequences comes from the case when $X_0, X_1,\ldots$ are independent and identically distributed (i.i.d.). Let $M_n=\max\{X_0,\ldots, X_{n-1}\}$. In this i.i.d.\ setting, it is clear that $\p(M_n\leq u)= (F(u))^n$, where $F$ is the d.f.\ of $X_0$, \ie $F(x):=\p(X_0\leq x)$. Hence, condition \eqref{eq:def-un} implies that
\[
\p(M_n\leq u_n)= (1-\p(X_0>u_n))^n\sim\left(1-\frac\tau n\right)^n\to\e^{-\tau},
\]
as $n\to\infty$. This means that the waiting times between exceedances of $u_n$ is approximately, exponentially distributed.

For example, if $\varphi$ is given as in \eqref{eq:observable-particular} and if $\mu$ has a density with respect to Lebesgue measure $m$,  where $\rho (\zeta):=\frac{d\mu}{dm} (\zeta),$ then the scaling constants can be chosen as  
$u_n =(1/d)\log n +\rho (\zeta)$.

\subsection{Rare Events Points Processes and respective convergence}

 Before we give the formal definition for REPP, we introduce some formalism. Let $\S$ denote the semi-ring of subsets of  $\R_0^+$ whose elements
are intervals of the type $[a,b)$, for $a,b\in\R_0^+$. Let $\RR$
denote the ring generated by $\S$. Recall that for every $J\in\RR$
there are $k\in\N$ and $k$ intervals $I_1,\ldots,I_k\in\S$ such that
$J=\cup_{i=1}^k I_j$. In order to fix notation, let
$a_j,b_j\in\R_0^+$ be such that $I_j=[a_j,b_j)\in\S$. For
$I=[a,b)\in\S$ and $\alpha\in \R$, we denote $\alpha I:=[\alpha
a,\alpha b)$ and $I+\alpha:=[a+\alpha,b+\alpha)$. Similarly, for
$J\in\RR$ define $\alpha J:=\alpha I_1\cup\cdots\cup \alpha I_k$ and
$J+\alpha:=(I_1+\alpha)\cup\cdots\cup (I_k+\alpha)$.

\begin{defn}
\label{def:REPP}
For stationary stochastic processes $X_0, X_1,\ldots$ and sequences $(u_n)_{n\in\N}$ satisfying \eqref{eq:def-un}, we define the \emph{rare event point process} (REPP) by
counting the number of exceedances (or hits to $U(u_n)$) during the (re-scaled) time period $v_nJ\in\RR$, where $J\in\RR$ and $v_n:=1/\mu(X_0>u_n)$ is, according to Kac's Theorem, the expected waiting time before the occurrence of one exceedance. To be more precise, for every $J\in\RR$, set
\begin{equation*}
 N_n(J):=
\sum_{j\in v_nJ\cap\N_0}\I_{X_j>u_n}.
\end{equation*}
\end{defn}

Our main result states that the REPP $N_n$ converges in distribution to a standard Poisson process. For the sake of completeness, we give next the meaning of convergence in distribution of point processes and also the definition of a standard Poisson process.  (See \cite{K86} for more details).

\begin{defn}
\label{def:convergence-point-processes}
Suppose that $(N_n)_{n\in\N}$ is a sequence of point processes defined on $\S$ and $N$ is another point process defined on $\S$. Then, we say that $N_n$ \emph{converges in distribution to} $N$ if the sequence of vector r.v.s $(N_n(J_1), \ldots, N_n(J_k))$ converges in distribution to $(N(J_1), \ldots, N(J_k))$, for every $k\in\N$ and all $J_1,\ldots, J_k\in\S$ such that $N(\partial J_i)=0$ a.s., for $i=1,\ldots,k$.
\end{defn}

\begin{defn}
\label{def:poisson-process}
Let $T_1, T_2,\ldots$ be  an i.i.d.\ sequence of random variables with common exponential distribution of mean $1/\theta$. Given this sequence of r.v., for $J\in\RR$, set
$$
N(J)=\int \I_J\;d\left(\sum_{i=1}^\infty \delta_{T_1+\ldots+T_i}\right),
$$ 
where $\delta_t$ denotes the Dirac measure at $t>0$.  We say that $N$ defined this way is a Poisson process of intensity $\theta$.
\end{defn}

To simplify the notation, whenever $J=[0,t)$ for some $t>0$ then we will write 
$$N_n(t):=N_n([0,t))) \quad \mbox{and}\quad N(t):=N([0,t)).$$

\begin{remark}
\label{rem:poisson-process}
If $\theta=1$ then we say that $N$ is a standard Poisson process and, for every $t>0$, the random variable $N(t)$ has a Poisson distribution of mean $t$. 
\end{remark}

\begin{remark}
\label{rem:EVL-HTS}
In the literature, the study of rare events is often tied to the existence of Extreme Value Laws (EVL) or the existence of Hitting Times Statistics (HTS) and Return Times Statistics (RTS). The existence of EVL has to do with the existence of distributional limits for $M_n=\max\{X_0, \ldots,X_{n-1}\}$. On the other hand, the existence of exponential HTS means the existence of a distributional limit for the elapsed time until the orbit hits certain balls around $\zeta$, when properly normalised. When the orbit starts in the target ball around $\zeta$ and consequently we look at the first return (rather than hit) and its limit distribution then we say we have RTS, instead. Since no exceedances of $u_n$ up to time $n$ means that there are no entrances in a certain ball around $\zeta$ up to time $n$, the existence of EVLs is equivalent to the existence of HTS (see \cite{FFT10,FFT11}). Moreover, in~\cite{HLV05} it was proved that an integral formula relates the 
distributions of HTS and RTS, and which in particular yields the standard exponential distribution
as its unique fixed point.  
Note that the convergence in distribution of the point processes $N_n$ to a standard Poisson process $N$ implies that  $\lim_{n\to\infty}\mu(M_n\leq u_n)=\lim_{n\to\infty}\mu(N_n(\tau)=0)=\e^{-\tau}$, which means that there exists an exponential EVL for $M_n$, which implies the existence of exponential HTS, for balls around $\zeta$, which in turn implies the existence of exponential RTS, for balls around $\zeta$.  We note also that certain extreme value statistics  lift from base transformations to 
suspension flows~\cite{HNT12}.
\end{remark}

Leadbetter \cite{L73} introduced some conditions on the dependence structure of general the stationary stochastic processes, called $D(u_n)$ and $D'(u_n)$, which can be used prove the convergence of REPP to the Poisson process (see \cite[Section~5]{LLR83}). However, condition $D(u_n)$, which imposes some sort of uniform mixing is often too strong to be verified in a dynamical setting. Recently, Freitas et al \cite{FFT10} gave an alternative condition, named $D_3(u_n)$, which together with the original $D'(u_n)$ was enough to prove the convergence of the REPP $N_n$ in distribution to the standard Poisson process $N$. This is precisely the statement of \cite[Theorem~5]{FFT10}. The great advantage of this weaker condition $D_3(u_n)$ is that it is much easier to check in a dynamical setting.

We will show that the stochastic processes arising from the billiard systems considered satisfy both these conditions $D_3(u_n)$ and $D'(u_n)$. Hence, we give next the precise formulation of the two conditions.

 For every $A\in\RR$ we define
\[
M(A):=\max\{X_i:i\in A\cap {\mathbb Z}\}.
\]
In the particular case where $A=[0,n)$ we simply write, as before,
$M_n=M([0,n)).$ Also note that $\{M(A)\leq u_n\}=\{N_n(v_n^{-1}A)=0\}$.

\begin{condition}[$D_3(u_n)$]\label{cond:D3}
We say that $D_3(u_n)$ holds for the sequence $X_0,X_1,\ldots$ if there exists
 $\gamma(n,t)$ nonincreasing in $t$ for each $n$ and
$n\gamma(n,t_n)\to0$ as $n\rightarrow\infty$ for some sequence
$t_n=o(n)$ (which means that $t_n/n\to0$ as $n\to \infty$) so that
\[ \left|\p\left(\{X_0>u_n\}\cap
  \{M(A+t)\leq u_n\}\right)-\p(X_0>u_n)
  \p(M(A)\leq u_n)\right|\leq \gamma(n,t),
\]
for all $A\in\RR$ and $t\in\N$.
\end{condition}

This condition is a sort of mixing requirement specially adapted to the problem of counting exceedances. Using decay of correlations of the billiard systems considered we will verify it for the stochastic processes arising from such systems. 

\begin{condition}[$D'(u_n)$]\label{cond:D'} We say that $D'(u_n)$
holds for the sequence $X_0$, $X_1$, $X_2$, $\ldots$ if 
\begin{equation}
\label{eq:D'un}
\lim_{k\to\infty}\limsup_{n\rightarrow\infty}\,n\sum_{j=1}^{\lfloor n/k\rfloor}\p( X_0>u_n,X_j>u_n)=0.
\end{equation}
\end{condition}

While $D_3(u_n)$ is a condition on the long range dependence structure of the stochastic process 
$X_0, X_1,\ldots$, $D'(u_n)$ is rather a condition on the short range dependence structure which 
inhibits the appearance of clusters of exceedances.
In other words, if we break the first $n$ random variables into blocks of size $\lfloor n/k\rfloor$, 
then $D'(u_n)$ restricts the existence of more than one exceedance in each block, which means 
that the exceedances should appear scattered through the time line.

\subsection{Planar dispersing billiards.}

Let $\Gamma = \set{\Gamma_i, i = 1:k}$ be a family of pairwise disjoint, simply connected $C^3$ curves with strictly positive curvature
on the two-dimensional torus $\mathbb{T}^2$.  The billiard flow  $B_t$ is the dynamical system 
generated  by the motion of  a point particle in $Q= \mathbb{T}^2/(\cup_{i=1}^k (\mbox{ interior } \Gamma_i))$ with constant unit velocity inside $Q$
and with elastic reflections at $\partial Q=\cup_{i=1}^k \Gamma_i$, where elastic means ``angle of incidence equals angle of reflection''. 
If each $\Gamma_i$ is a circle then this system is called a periodic Lorentz gas, a well-studied model in physics. The billiard
flow is Hamiltonian and preserves  a probability measure (which is Liouville measure) $\tilde{\mu}$ given by
 $d\tilde{\mu}=C_{Q}\,dq\,dt$ where $C_{Q}$ is a normalizing constant and $q\in Q$, $t\in \R$ are Euclidean coordinates. 

We first consider the billiard map  $T: \partial Q \to \partial Q$.  Let $r$ be a one-dimensional coordinatization of 
$\Gamma$ corresponding to length and let $n(r)$ be the outward normal to $\Gamma$ at the point $r$. For each
$r\in \Gamma$ we consider the tangent
space at $r$ consisting of unit vectors $v$ such that $(n(r),v)\ge 0$. We identify each such unit vector $v$ with an 
angle $\theta \in [-\pi/2, \pi/2]$. The boundary $M$ is then parametrized by $M:=\partial Q=\Gamma\times [-\pi/2, \pi/2]$  so that $M$ consists of  the 
points $(r,\theta)$. $T:M\to M$ is the Poincar\'e map that gives the position and angle $T(r,\theta)=(r_1,\theta_1)$  after a point $(r,\theta)$
flows under $B_t$ and collides again with $M$, according to  the rule  angle of incidence equals angle of reflection. Thus if 
$ (r,\theta)$ is the time of flight before collision $T(r,\theta)=B_{h (r,\theta)} (r,\theta)$. 
The billiard map preserves a measure $d\mu=c_{M} \cos \theta\, dr d\theta$ equivalent to the 
$2$-dimensional Lebesgue measure $dm=dr\,d\theta$ with
density $\rho (x)=C_M\log\theta$ where $x=(r,\theta)$.

Under the assumption of finite  horizon condition, namely, 
that the time of flight $h(r,\theta)$  is bounded above, Young~\cite{Y98} proved that the  billiard map has exponential decay of correlations for H\"{o}lder observations. The strategy relied on building a Gibbs-Markov structure, that is now usually called \emph{Young Tower}, with a corresponding induced map bearing nice hyperbolic properties. Then the idea was to pass the good statistical properties of the induced map to the original system, in which the tail of the inducing time ended up playing a prominent role, in particular, in the determination of the system's mixing rates. This settled a long-standing question about the rate of decay of correlations in such systems.
Chernov~\cite{C99} extended this result to planar dispersing billiards with infinite horizon where $h(x,r)<\infty$ for 
all but finitely many points $(r,\theta)$ but is not essentially bounded.   Chernov also 
proved exponential decay for  dispersing billiards with corner points (a class of billiards we do not discuss in this paper). 
A good reference for background results for this section are the papers~\cite{BSC90,BSC91,Y98, C99}.

Our  first theorem is:

\begin{thm}\label{Poisson}
Let $T:M\to M$ be a  planar dispersing  billiard map.  Consider that the stochastic process $X_0, X_1,\ldots$ is given as in \eqref{eq:def-stat-stoch-proc-DS} for the type of observables $\varphi$ considered above.
Then for $\mu$ a.e. $\zeta$, conditions $D_3(u_n)$ and $D'(u_n)$ hold for $X_0, X_1,\ldots$ and sequences $(u_n)_{n\in\N}$ satisfying \eqref{eq:def-un}. Consequently, the REPP $N_n$ given in Definition~\ref{def:REPP} converges in distribution to the standard Poisson process.  
\end{thm} 

\begin{remark}
In particular, note that for each $t> 0$ and each integer $k$ we have:
\[
\lim_{n\to \infty} \mu (N_n (t) =k)=e^{-t}\frac{t^k}{k!}.
\]
\end{remark}

The strategy for proving Theorem~\ref{Poisson} is to show the validity of conditions $D_3(u_n)$ and 
$D'(u_n)$ for various dynamical systems modelled by Young Towers, in particular dispersing planar
billiards. The proof of $D'(u_n)$ has been given in Gupta et al~\cite{GHN11} but we reproduce it for completeness in Section~\ref{sec:d'.sing}. 
The proof of $D_3 (u_n)$ is similar to that of the proof for a related condition $D_2 (u_n)$ (useful in establishing the existence of EVL)
given in~\cite{GHN11}.

\subsection{Billiards with polynomial mixing rates}

In \cite{Y98}, Young introduced a Gibbs-Markov structure (which became known as Young tower) which she used to study dispersing billiards with exponential decay of correlations. Later on, Markarian~\cite{M04} developed an elegant technique to use inducing to establish polynomial upper bounds for rates of decay of correlation in 
 certain billiard systems. Young~\cite{Y99}  had used coupling to establish polynomial decay for certain non-uniformly expanding maps and 
 Markarian's ideas built upon this work.
 
  Markarian's idea was to find  a subset $M\subset X$ on which the first return map $F: M\to M$ has strong hyperbolic behavior, in particular
 admits a Young Tower with exponential tails. His approach was subsequently extended  by Chernov
  and Zhang~\cite{CZ05} to many billiard systems exhibiting polynomial decay.

\textbf{Notation:}  Given a finite measure $\mu$ on $X$ and a measurable set $A\subset X$ ($\mu(A)>0$),
 then we denote by $\mu_A$ the corresponding conditional measure on $A$, i.e. 
 $\mu_A(B)=\mu(A\cap B)/\mu(A)$ for $B\subset X$ measurable.

  The first hitting time function is given by 
 \begin{equation}
 \label{eq:hitting-time}
 r_M(x):=\min \{ j\ge 1: T^j (x) \in M\}
 \end{equation}
 and measures the time until the orbit of a point $x\in X$ enters $M$. The induced map
 $F: M\circlearrowleft$ is then given by $F=T^{r_M}$ and its
 invariant measure is the normalised measure $\mu_M$.  
 If the return time tails decay polynomially, that is if $\mu (x\in X: r_M (x) > n) =\mathcal{O}(n^{-a})$
  for some constant $a>0$ then  Markarian~\cite{M04} showed that
  \begin{equation}
  \label{eq:polynomial-DC}
 \left|\int\phi \, \psi\circ T^n  d\mu -\int \phi d\mu \int \psi d\mu \right| \le C n^{-a} \|\phi\|_{Lip} \|\psi \|_{Lip}
 \end{equation}
 for some constant $C$. This allows us to extend our results above on Poisson limit laws to the 
 setting of billiards with polynomial mixing rates, by first inducing on $M$
 and then realizing $T: X\to X$ as a first-return time Tower over $(F,M,\mu_M)$. 

 \begin{thm}\label{poly}
  Suppose $(T,X,\mu)$  is a billiard system with SRB measure $\mu$  and $M\subset X$ is a subset such that the first return map $F: M \to M$ admits the structure of a 
 Young Tower with exponential tails. Suppose further that the function $\tau$, defined in \eqref{eq:hitting-time}, is integrable with respect to $\mu$. Consider now that the stochastic process $X_0, X_1,\ldots$ is given as in \eqref{eq:def-stat-stoch-proc-DS} for the type of observables $\varphi$ considered above.
Then for $\mu$ a.e. $\zeta$, the REPP $N_n$ given in Definition~\ref{def:REPP} converges in distribution to the standard Poisson process.  
\end{thm}

The idea to prove Theorem~\ref{poly} is to use the same strategy for dispersing billiards to show that for the first return time map $F:M\to M$ and for the stochastic processes it gives rise to, we have convergence of the points processes $N_n$ to the standard Poisson process, $\mu$-a.e. Then we use an idea introduced in \cite{BSTV03}, which essentially says that the original system $T$ shares the same property of the first return time map $F$, meaning that for stochastic processes arising from the dynamics of $T$ we also have that the points processes $N_n$ converge to the standard Poisson process, for $\mu$-a.e. $\zeta$. Unfortunately, the original statement of \cite{BSTV03} only allows to conclude that if the first return time $F$ has exponential HTS/RTS for balls around $\mu$-a.e. $\zeta$ then the original system $T$ also has the same property. However, as remarked in \cite{BSTV03}, a small adjustment to the argument used there allows to prove the stronger statement that the same holds for the convergence of point processes to the standard Poisson process. For completeness, we state here such a result  and prove it in Section~\ref{sec:polynomial-rates}.

In order to distinguish objects of the induced system $F$ from the corresponding objects of the original system, we will use the symbol $\hat\cdot$ over these objects. In particular we will write $\hat\mu:=\mu_M$. Let $\zeta\in M$ and $\varphi$ be an observable as above, which achieves a global maximum at $\zeta$.

This new induced system gives rise to a new set of random variables
$$\hat X_n=\varphi\circ F^n.$$
 We can thus consider $\hat N_n(J)$ for $J\in \S$ and $\hat v=1/\hat \mu(\hat X_0>u_n)$ defined analogously to $N_n(J)$ in Definition~\ref{def:REPP} for the original system.

   \begin{prop}\label{Extend}
 Suppose $(T,X,\mu)$ is an dynamical system with $\mu$ absolutely continuous with respect to Lebesgue, $M\subset X$ is a measurable set with $\mu(M)>0$ and let $F:M\to M$ denote the  first return induced map. Assume that for $\hat N_n$ converges in distribution (w.r.t. $\hat \mu$) to a standard Poisson process $N$, for $\hat\mu$-a.e. $\zeta\in M$. 
Then for the original map
 $(T,X,\mu)$  we can say that $N_n$ converges in distribution (w.r.t.\ the measure $\mu$) to a standard Poisson process for $\mu$-a.e. point  $\zeta \in M$.
 \end{prop}

 We remark that the statement of \cite{BSTV03}, which said that the limit distribution for HTS/RTS for the induced map $F$  was equal, at $\mu$-a.e. point $\zeta$, to the respective HTS/RTS distributional limit for the original system $T$, was extended in~\cite{HWZ} by removing the $\mu$-a.e. point $\zeta$ restriction. In an ongoing work about an extremal dichotomy for intermittent maps, the first named author with A.C.M. Freitas, M. Todd and S. Vaienti proved an extension of the~\cite{HWZ} result to include the convergence of point processes, which implies Proposition~\ref{Extend}.

\section{Condition $D_3 (u_n)$ for Young Towers with exponential tails}

We will make an assumption on the invariant measure $\mu$, which is automatically satisfied for planar  billiard maps. We assume,

\noindent {\bf Assumption A~\label{annulus}:} 
 For $\mu$-a.e.  $\zeta\in M $ there exists $\xi:=\xi(\zeta)>0$ such that if 
 $A_{r,\epsilon}(\zeta)=\{ y \in M : r\le d(\zeta,y) \le r+\epsilon  \}$ is a shell   of inner  radius $r$
 and outer radius  $r+ \epsilon$  about  the point $\zeta$  and  if  $r$ sufficiently small,  $0<\epsilon\ll r<1$,  
 then $\mu(A_{r,\epsilon} (\zeta)) \le \epsilon^\xi$.

 \vspace{2mm}
 
 Assumption A is satisfied by planar dispersing billiards with finite and infinite horizon as the 
 invariant measure is equivalent to Lebesgue.  This is proved in~\cite[Appendix 2]{BSC91} 
 where it is shown that $\xi$ may be taken as $1$ in the case of 
finite horizon and $4/5$ in the case of infinite horizon.

 The Young Tower
assumption implies that there exists a subset $\Lambda\subset M$ such that $\Lambda$ has a  
hyperbolic product structure and that (P1)--(P4) of~\cite{Y98} hold. We refer the reader to Young's paper~\cite{Y98}
and the book by Baladi~\cite{B00} for details. A similar axiomatic construction of a tower is given by Chernov~\cite{C99}
which is  a good reference for background on  dispersing billiard maps and flows. 

By taking $T$ to be a local diffeomorphism we allow the map $T$  or its derivative 
to have discontinuities or  singularities. 

Next  we describe briefly the structure of  a Young Tower with exponential return time  tails for  a local diffeomorphism $T:M\to M$
of a  Riemannian   manifold  $M$ equipped with Lebesgue measure $m$.  
 
There is a set $\Lambda$ with a hyperbolic product structure as in Young~\cite{Y98}
and assume there is an $\mathscr{L}^1(m)$ return time function $R:\Delta_0\to\mathbb{N}$.
Moreover assume there is a countable partition $\Lambda_{0,i}$ of $\Delta_0$ so that 
$R$ is constant on each partition element $\Lambda_{0,i}$. We put $R_i:=R|_{\Lambda_{0, i}}$.
Now  the  Young Tower is defined by
\[
\Delta = \cup_{i, l \le R_i-1}\{ (x, l) : x\in \Lambda_{0,i }\}
\] 
and the  tower map $F:\Delta \to \Delta $ by
\[
F(x, l) = \begin{cases}(x, l+1) & \mbox{ if }  x\in \Lambda_{0,i}, l<R_i-1\\
(T^{R_i}x, 0) &\mbox{ if } x\in \Lambda_{0,i}, l = R_i-1\end{cases}.
\]
We will refer to $\Delta_0:= \cup_i(\Lambda_{0,i} ,0)$ as  the  base of the tower $\Delta$  and denote 
 $\Lambda_i := \Lambda_{0,i}$. Similarly we call $\Delta_l=\{(x,l): l<R(x) \}$, the $l$th level of the tower. 
Define the return map $f=T^R:\Delta_0\to\Delta_0$ by $f(x)=T^{R(x)} (x)$. We may form  a quotiented tower
(see~\cite{Y98} for details) by introducing an equivalence relation for points on the same stable manifold. We now  list the features 
of the Tower that we will use.

There exists an invariant measure $m_0$  for $f: \Delta_0\to \Delta_0$ which has 
absolutely continuous conditional measures on local unstable manifolds in $\Delta_0$, with density bounded uniformly from above and below.

There exists an $F$-invariant measure $\nu$ on $\Delta$ which is given by
$\nu (B)=\frac{m_0(F^{-l} B)}{\int_{\Lambda_0} R\,dm_0 }$ 
 for measurable $B\subset \Lambda_{l}$, and extended to the entire tower $\Delta$ in the obvious way.
There is a projection $\pi:\Delta \to M$ given by $\pi(x, l) = T^l(x)$ which semi-conjugates $F$ and $T$, 
that is it satisfies $\pi\circ F = T\circ\pi$. 
The invariant measure  $\mu$, which is an SRB measure for $T: M\to M$,  is then given by $\mu = \pi_*\nu$.
Denote by $W^s_{loc} (x)$ the local stable manifold through $x$ i.e.\
there exists $\epsilon(x)>0$ and $C>0$, $0<\alpha<1$ such that 
$$
W^s_{loc}= \{ y: d(x,y)<\epsilon, d(T^n y,T^n x)< C \alpha^n\;\forall n\ge 0\}.
$$  
 We use  the notation $W^s_{loc}$ rather than $W^s_{\epsilon} (x)$ in  contexts where  the length of the local stable manifold is not   important. Analogously one defines the local unstable manifold $W^u_{loc} (x)$.
Let $B(x,r)$  denote the ball of radius $r$ centered at the point $x$. 
 We lift a function $\phi: M \to \R$ to $\Delta$ by defining, with abuse of notation, $\phi(x,l)=\phi(T^l x)$.

Under the assumption of exponential tails, that is if $m(R>n)=\mathcal{O}( \theta^n)$ for some $0<\theta <1$ then from the computations in \cite{Y98} one can deduce that for all Lipschitz $\phi$, $\psi$ we have
 \begin{equation}\label{Lip_Lip_decay}
 \left |\int \phi  \psi\circ T^n  d\mu -\int \phi d\mu \int \psi d\mu \right| \le C \theta_1^n \|\phi\|_{Lip} \|\psi \|_{Lip}
 \end{equation}
 for some constant $C$. Moreover, if  the lift of $\psi$ is constant on local stable leaves of the Young Tower,
 then
 \begin{equation}\label{Lip_infinity_decay}
 \left|\int \phi  \psi\circ T^n  \,d\mu -\int \phi \,d\mu \int \psi \,d\mu \right| 
 \le C \theta_1^n \|\phi\|_{Lip} \|\psi \|_{\infty}.
 \end{equation}

\vspace{3mm}

As before, let $\zeta$  be in the support of $\mu$
 and define a stochastic process $X_n$ given by $X_n(x) = -\log d(T^n x, \zeta)$. 
In the remainder of this section we establish condition $D_3(u_n)$ for maps modeled by  a 
Young Tower with exponential tails satisfying Assumption A.  Our main theorem for this section is:

\begin{thm}\label{thm:dun}
Let $T:(M,\mu) \to (M,\mu) $ be a dynamical system  modeled by a Young Tower
with exponential tails satisfying Assumption A.
Then the stochastic process $X_0, X_1,\ldots$ defined as in \eqref{eq:def-stat-stoch-proc-DS} 
satisfies the condition $D_3(u_n)$. 
\end{thm}

\begin{proof}
 We first define 
 \[
 B_{r,k}(\zeta) = \left\{x:T^{k}(W^s_\eta(x))\cap \partial B (\zeta, r)\ne\emptyset\right\}.
 \]
and obtain as an immediate consequence of Assumption A the following:

\begin{prop}\label{prop:annulus1}
Under Assumption A there exist constants $C>0$ and $0<\tau_1<1$ such that for any $r, k$
 \begin{equation}\label{annulus}
  \mu(B_{r, k}(\zeta))\le  C \tau_1^{k}.
 \end{equation}
\end{prop}

\begin{proof}
 As a consequence of uniform contraction of local stable manifolds~\cite[(P2)]{Y98}, 
there exist $\alpha\in(0, 1)$ and $c_1>0$ such that $d(T^n(x), T^n(y))\le c_1\alpha^n$ for all
 $y\in W^s_\eta(x).$ In particular, this implies that $|T^k(W_\eta^s(x))|\le c_1\alpha^k$ where $|\cdot|$ denotes the length with respect to the Lebesgue measure. Therefore,  for every  $x\in B_{r,k}(\zeta)$  
 the leaf $T^k(W^s_\eta(x))$ lies in an annulus of width $2c_1\alpha^k$ around
 $\partial B(\zeta,r)$.  By Assumption A and invariance of $\mu$ the result follows.
\end{proof}

We now continue the proof of Theorem~\ref{thm:dun}. The constant $\tau_1$ below is from 
Proposition~\ref{prop:annulus1}.
Let $A\in S$, so that $A=\cup_{j=1}^{l}[a_j,b_j)$ and define $I_A=[a_1,b_l]$.

\begin{lemma}\label{lemma:dun-prelim}
Suppose
 $\Phi:M\to \R$ is Lipschitz and $\Psi_{A}$ is the indicator function 
 \[
\Psi_{A}:= \I_{\{M(A)\le u_n\}}
 \]
Then for all $j\geq 0$

\begin{equation}
 \left|\int\Phi\Psi_{A}\circ T^j \text{d}\mu - \int\Phi\text{d}\mu\int\Psi_{A}\text{d}\mu\right|\le \C\left(\|\Phi\|_\infty \tau_1^{\floor{j/2}}+\|\Phi\|_{\text{Lip}}\theta^{\floor{j/2}}\right).
\end{equation}

\end{lemma}

\begin{proof}
Define the function $\tilde\Phi:\Delta\to\mathbb{R}$ by $\tilde\Phi(x,r) = \Phi(T^r(x))$ 
and the function $\tilde \Psi_{A}(x,r) = \Psi_{A}(T^r(x)).$ 
We choose a reference unstable manifold $\tilde\gamma^u\subset \Delta_0$ and by the hyperbolic product 
structure each local stable manifold $W^s_\eta(x)$ will intersect $\tilde\gamma^u$ 
in a unique point $\hat{x}$.
Here $x$ denotes a point in the base of the tower $\Delta_0$ and we therefore have  
$x\in W^s_\eta(\hat x).$

We define the function $\overline\Psi_{A}(x,r):=\Psi_{A}(\hat x, r)$. 
We note that $\overline \Psi_{A}$ is constant along stable manifolds in $\Delta$ and the set of points 
where $\overline \Psi_{A}\neq\tilde\Psi_{A}$ is, by definition, the set of $(x,r)$  which project to points $T^r (x)$ for which there 
exist $x_1,x_2$ on the same local stable manifold as $T^r (x)$ for which 
\[
x_1\in \{M(A)\le u_n\}
\]
 but 
\[
x_2 \notin \{M(A)\le u_n \}
\]
This set is contained inside $\cup_{k =a_1}^{a_1+b_l}T^{-k}B_{u_n,k}$. If
we let $a_1\ge\floor{j/2}$  then by Proposition~\ref{prop:annulus1}
we  have
\[\nu\set{\tilde\Psi_{\floor{j/2},\floor{j/2}+b_l}\neq\overline\Psi_{\floor{j/2},\floor{j/2}+b_l}} \le 
\sum_{k=\floor{j/2}}^{\infty}\mu(B_{u_n, k})\le \C \tau_1^{\floor{j/2}}.\] 
By the decay of correlations as proved in~\cite{Y98} under the assumption of exponential tails, we have
\[
 \abs{\int\tilde\Phi\overline\Psi_{A+\floor{j/2}}\circ F^{j -\floor{j/2}}d\nu - \int\tilde\Phi d\nu\int \overline\Psi_{A+\floor{j/2}}d\nu}\le \C\|\Phi\|_{\text{Lip}}\|\Psi\|_{\infty}\theta^{\floor{j/2}}.
\] 
Recall,
\[
\abs{\int \Phi\Psi_{A+\floor{j/2}}\circ T^{j -\floor{j/2}}d\mu - \int\Phi d\nu\int \Psi_{A+\floor{j/2}}d\mu}
= \abs{\int\tilde\Phi\tilde\Psi_{A+\floor{j/2}}\circ F^{j -\floor{j/2}}d\nu - \int\tilde\Phi d\nu\int \tilde\Psi_{A+\floor{j/2}}d\nu}
\]
We will use the identity $\int \tilde{\phi}\tilde{\psi}\circ F-\int \tilde{\phi} \int \tilde{\psi}= \int \tilde{\phi} (\tilde{\psi}\circ F-\bar{\psi}\circ F)+
\int \tilde{\phi}\bar{\psi}\circ F -\int \tilde{\phi} \int \bar{\psi} +\int \tilde{\phi} \int \bar{\psi} -\int \tilde{\phi} \int \tilde{\psi}$.
 Thus 
\begin{align}
 \abs{\int \Phi\Psi_{A+\floor{j/2}}\circ T^{j -\floor{j/2}}d\mu - \int\Phi d\nu\int \Psi_{A+\floor{j/2}}d\mu}\nonumber\\
= \abs{\int\tilde\Phi\tilde\Psi_{A+\floor{j/2}}\circ F^{j -\floor{j/2}}d\nu - \int\tilde\Phi d\nu\int \tilde\Psi_{A+\floor{j/2}}d\nu}\nonumber\\
\le \abs{\int \tilde\Phi\left(\tilde\Psi_{A+\floor{j/2}} - \overline\Psi_{A+\floor{j/2}}\right)\circ F^{j - \floor{j/2}}d\nu}  +\C \|\Phi\|_{\text{Lip}}\theta^{\floor{j/2}} \nonumber\\ 
+ \abs{\int\tilde\Phi d\nu\int\left(\overline\Psi_{A+\floor{j/2}} - \tilde\Psi_{A+\floor{j/2}}\right)
d\nu}\nonumber\\
\le \C\left(2\|\Phi\|_\infty\nu\set{\overline\Psi_{A+\floor{j/2}}\neq\tilde\Psi_{A+\floor{j/2}}}+\|\Phi\|_{\text{Lip}}\theta^{\floor{j/2}}\right)\nonumber\\
\le \C\left(\|\Phi\|_\infty  \tau_1^{\floor{j/2}}+\|\Phi\|_{\text{Lip}}\theta^{\floor{j/2}}\right).
\end{align}
We complete the proof by observing that $\int\Psi_{A} d\mu = \int\Psi_{A+\floor{j/2}} d\mu$ by the $\mu$ invariance of $T$ and that
 $\Psi_{A+\floor{j/2}}\circ T^{j -\floor{j/2}} = \Psi_{A+j} = \Psi_{A}\circ T^{j}.$
\end{proof}

To prove condition $D_3(u_n)$, we will approximate the characteristic function of  the set 
$U_n=\{ X_0>  u_n\} $ by a suitable Lipschitz function. This approximation will decrease sharply to zero near the boundary of  the set $U_n$.  The bound in Lemma~\ref{lemma:dun-prelim} involves the Lipschitz norm, therefore, we need to be able to bound the increase in this norm.

 We approximate the indicator function $\I_{U_n}$ by a Lipschitz continuous function $\Phi_n$ as follows. 
 Since $U_n$ is a ball of some radius $r_n \sim \frac{1}{\sqrt{n}}$ centered at the point $\zeta$,
  we define $\Phi_n$ to be $1$ inside a ball centered at $\zeta$ of radius $r_n-n^{-\frac{2}{\xi}}$, where $\xi$ comes from Assumption A and decaying to $0$ so that on the boundary of
  $U_n$, $\Phi_n$ vanishes. The Lipschitz norm of $\Phi_n$ is seen to be bounded by $n^{\frac{2}{\xi}}$
  and $\|\I_{U_n}- \Phi_n \|_1\le \frac{1}{n^2}$. Therefore
\begin{align}
 \abs{\int \I_{U_n} \Psi_{A+\floor{j/2}}\circ T^{j -\floor{j/2}}d\mu - \mu(U_n)\int \Psi_{A+l}d\mu}\nonumber\\
\le \abs{\int \left(\I_{U_n} - \Phi_n\right)\Psi_{A+\floor{j/2}}d\mu}+\C \left(\|\Phi_n \|_\infty j^2\tau_1^{\floor{j/4}}+\|\Phi_n \|_{\text{Lip}}\theta^{\floor{j/2}}\right)\nonumber\\
+ \abs{\int \left(\I_{U_n} - \Phi_n \right)d\mu \int \Psi_{A+\floor{j/2}}d\mu}, 
\end{align}
and consequently
\[
 \abs{\mu(U_n\cap \{ M(A+l)\le  u_n\} ) - \mu(U_n)\mu(\{M(A)\le  u_n \} )}\le \gamma(n,j)
\] where 
\[
 \gamma(n,j) = \C\left(n^{-2}+ n^{\frac{2}{\xi}} \theta_1^{\floor{j/2}}\right)
\] where $\theta_1 = \max\set{\tau_1, \theta}. $
Let $j=t_n= (\log n)^{5}$. Then  $n\gamma(n, t_n)\to 0$ as $n\to\infty.$  Note that we had considerable 
freedom of choice of $t_n$, anticipating our applications we choose $t_n=(\log n)^{5}$.
\end{proof}

\subsection{Property $D'(u_n)$ for Planar Dispersing  Billiard Maps}\label{sec:d'.sing}
We have shown  $D_3 (u_n)$ is immediate in  the case of dispersing billiard maps with finite horizon, as they are modeled
by a Young Tower in~\cite{Y98} and  have exponentially decaying correlations.  Chernov~\cite[Section 5]{C99} (see also~\cite[Section 5]{BSC91}) constructs a Young Tower  for  billiards with infinite horizon to prove exponential decay of correlations so that condition  $D_3(u_n)$ is satisfied by this class of billiard map as well.  Hence to prove a Poisson limit law we need only prove  condition $D'(u_n)$, which we do in  this section.

It is known (see~\cite[Lemma 7.1]{C99} for finite horizon and ~\cite[Section 8]{C99} for infinite horizon) 
that dispersing billiard maps expand in the unstable direction 
 in the Euclidean metric $|\cdot|=\sqrt{(dr)^2+(d\phi)^2}$, in that  $|DT_u^n v|\ge C \tilde{\lambda}^n |v|$
for some constant $C>0$ and $\tilde{\lambda}>1$  independent of $v$. 

If we choose $N_0$ so that $\lambda:= C\tilde{ \lambda}^{N_0}>1$ then $T^{N_0} $ (or $DT^{N_0}$) 
expands unstable manifolds (tangent vectors to unstable manifolds) uniformly in the Euclidean
metric.

It is common to use the $p$-metric in proving ergodic properties of   billiards. 
Recall that
for any  curve $\gamma$, the $p$-norm of a tangent vector to $\gamma$ is given as $|v|_p =\cos\phi(r)|dr|$ where $\gamma$ is parametrized in the $(r, \phi)$ plane as $(r, \phi(r)).$ Since the Euclidean metric in the
 $(r, \phi)$ plane is given by $ds^2 = dr^2+d\phi^2$ this implies that $|v|_p\le \cos\phi(r)\,ds\le ds = |v|$.  We will use $l_p (C)$ to denote the length of a 
curve in  the $p$-metric and $l(C)$ to denote length in the Euclidean metric.  
If $\gamma$ is a local unstable manifold or local stable manifold then  
$C_1 l (\gamma )_p \le l(\gamma) \le C_2 \sqrt{ l_p (\gamma ) }$.

For planar dispersing  billiards
there exists an invariant measure $\mu$ (which is equivalent to the 2-dimensional Lebesgue measure)
and through $\mu$-a.e. point $x$ there exists a local stable manifold $W_{loc}^s (x)$ and a local 
unstable manifold $W_{loc}^u (x)$.  The SRB measure $\mu$ has absolutely continuous (with respect to Lebesgue measure) conditional measures $\mu_x$ on each $W_{loc}^u (x)$. 
The expansion by $DT$ is unbounded however in the $p$-metric
at $\cos\theta=0$ and this may lead to  quite different expansion rates at different points on $W_{loc}^u (x)$. To overcome this
effect and obtain uniform estimates on the densities of conditional SRB measure it is common to definite homogeneous local
unstable and local stable manifolds. This approach was adopted in~\cite{BSC90,BSC91,C99,Y98}. Fix a large $k_0$ and 
define for $k>k_0$
\[
I_k=\left\{(r,\theta): \frac{\pi}{2} -k^{-2} <\theta<\frac{\pi}{2}-(k+1)^{-2} \right\}
\]
\[
I_{-k}=\left\{(r,\theta): -\frac{\pi}{2} +(k+1)^{-2} <\theta<-\frac{\pi}{2}+k^{-2}\right \}
\]
and 
\[
I_{k_0}=\left\{(r,\theta): -\frac{\pi}{2} +k_0^{-2} <\theta<\frac{\pi}{2}- k_0^{-2} \right\}.
\]
We call a local unstable (stable) manifold $W^u_{loc} (x)$, ($W^s_{loc} (x)$) 
{\em homogeneous} if  $T^n  W^u_{loc} (x)$  ($T^{-n}  W^s_{loc} (x)$) does not intersect any of the line segments in $\cup_{k>k_0} (I_k\cup I_{-k})\cup I_{k_0}$ for all $n\ge 0$. Homogeneous
$W^u_{loc} (x)$ have almost constant  conditional SRB densities $\frac{d\mu_x}{dm_x}$ in the sense that 
there exists $C>0$ such that $\frac{1}{C} \le \frac{d\mu_x}{dm_x}(z_1) /\frac{d\mu_x}{dm_x}(z_2) \le C$ for all $z_1,~z_2 \in W^u_{loc} (x)$ (see ~\cite[Section 2]{C99} and the remarks following Theorem 3.1).

From this point on all the local unstable (stable) manifolds that we consider will be homogeneous. 
Bunimovich et al~\cite[Appendix 2, Equation A2.1]{BSC91}  give quantitative estimates on the 
length of  homogeneous  $W^u_{loc} (x)$. They  show  there exists 
$C,~\tau >0$ such that  
$\mu \{ x: l(W_{loc}^s (x))<\epsilon \mbox{ or }  l(W_{loc}^u (x))<\epsilon \} \le C \epsilon^{\tau}$ 
where   $l(C)$ denotes 1-dimensional  Lebesgue measure or length of a rectifiable curve $C$.  
In our setting $\tau$ could be taken to be  $\frac{2}{9}$, its exact value will play no role but for 
simplicity in the forthcoming estimates  we assume $0<\tau<\frac{1}{2}$.

The natural measure $\mu$
has absolutely continuous conditional measures $\mu_x$ on local unstable manifolds $W_{loc}^u (x)$ which have almost uniform
densities with respect to Lebesgue measure on $W_{loc}^u (x)$ by~\cite[Equation 2.4]{C99}.

\subsubsection{Controlling the measure of the set of rapidly returning points.}
 Let $A_{\sqrt{\epsilon}}=\{ x: |W_{loc}^u (x)|>\sqrt{\epsilon}\}$ then 
 $\mu (A^c_{\sqrt{\epsilon}})< c_1\epsilon^{\tau/2}$ by Bunimovich's result.
  Let $x\in A_{\sqrt{\epsilon}}$
and consider $W_{loc}^u (x)$. Since $|T^{-k} W_{loc}^u (x)|<\lambda^{-1}  |W_{loc}^u (x)|$ for $k>N_0$
we obtain by the triangle inequality for $y,y'\in W_{loc}^u (x)$:
$$
d(y,y')\le d(T^{-k}y',y')+d(T^{-k}y,T^{-k}y')+d(T^{-k}y,y)
\le 2\epsilon+\frac1\lambda d(y,y')
$$
which implies $d(y,y')\le2(1-\frac1\lambda)\epsilon$. 
Thus
$$
l\{y\in  W_{loc}^u (x): d(y,T^{-k} y)<\epsilon\}\le 2(1-\lambda^{-1})\epsilon
\le c_2\sqrt{\epsilon}\, l\{y\in  W_{loc}^u (x)\}.
$$
Since the density of the conditional SRB-measure  $\mu_x$ is bounded above and below with respect
 to one-dimensional Lebesgue measure we obtain 
 $\mu_x (y\in  W_{loc}^u (x): d(y,T^{-k} y)<\epsilon)< c_3\sqrt{\epsilon}$. Integrating over all unstable manifolds in $A_{\sqrt{\epsilon}}$ (throwing away
the set $\mu (A^c_{\sqrt{\epsilon}})$)
we obtain $\mu \{ x: d(T^{-k}x, x)<\epsilon) <c_4 \epsilon^{\tau/2}$ ($c_4\le c_1+c_3$).
 Since $\mu$ is $T$-invariant we get
$$
\mathcal{E}_{k} (\epsilon):= \mu \{ x: d(T^k x, x)<\epsilon \} < c_4 \epsilon^{\tau/2}
$$
for $k>N_0$. 
Consequently
\[
E_k:= \{x: d(T^j x,x)\le \frac{2}{\sqrt{k}}~\mbox{ for some }~1\le j \le \log^5 k \}
\]
obeys the upper bound $\mu (E_k)\le c_5k^{-\sigma}$ for any $\sigma>\frac\tau4$.
Let us note that a similar result has been shown in~\cite{CC13}, Lemma~4.1.

\subsubsection{Controlling the measure of the set of points whose neighborhoods have large overlaps with the sets $E_k$.}
As in \cite{C01}, we define the Hardy-Littlewood maximal function  $\mathcal{M}_l$ for $\phi(x)= 1_{E_l} (x)\rho(x)$ where 
$\rho(x)=\frac{d\mu}{dm} (x)$, so that 
\[
 \mathcal{M}_l(x):=\sup_{a>0}\frac{1}{m(B(x,a))}\int_{B(x,a)} 1_{E_l}(y)\rho(y)\,dm(y).
\]
Hence (cf.~\cite[Page 96]{F99})
\[
m( |\mathcal{M}_l|>C)\le \frac{\|1_{E_l} \rho \|_1}{C}
\]
where $\|\cdot\|_1$ is the $\mathscr{L}^1$ norm with respect to $m$. 
Let 
\[
F_k:=\{ x: \mu (B (x,k^{-\gamma/2})\cap E_{k^{\gamma/2}})\ge  (k^{-\gamma\beta /2} )k^{-\gamma}
\]
Then $F_{k}\subset \{\mathcal{M}_{k^{\gamma/2}}>k^{-\gamma\beta /2} \}$ and hence
\[
m(F_k) \le \mu (E_{k^{\gamma/2}})k^{\gamma\beta /2}\le C k^{-\gamma\sigma/2}k^{\gamma\beta /2}.
\]
If we take $0<\beta <\sigma/2$ and $\gamma >\sigma/4$ then  for some $\delta>0$, 
$ k^{-\gamma\sigma/2}k^{\gamma\beta /2}<k^{-1-\delta}$ and hence 
\[
\sum_k m(F_k)<\infty.
\]
Thus by Borel-Cantelli  for $m$-a.e.\ (hence $\mu$-a.e.) $\zeta\in X$  there exists $N(\zeta)$ such that 
$\zeta \not \in F_k$ for all $k>N(\zeta)$. 
Thus along the subsequence $n_k=k^{-\gamma/2}$,  
$\mu (U_{n_k}\cap T^{-j}U_{n_k}) \le n_k^{-1-\delta}$ for $k>N(\zeta)$ 
where as before $U_n=\{X_0>u_n\}$ (and thus $T^{-j}U_n=\{X_0\circ T^j>u_n\}$).
This is sufficient to obtain an estimate for  all $u_n$.   For if $k^{\gamma/2} \le n \le (k+1)^{\gamma/2}$  
then
$\mu (U_n\cap T^{-j}U_n) \le \mu (U_{n_k}\cap T^{-j}U_{n_k})\le n_k^{-1-\delta} \le  2n^{-1-\delta}$ for all $n$ large enough as $(\frac{k+1}{k})^{\gamma/2}\to1$.

We now control the iterates $1\le j \le N_0$. If $\zeta$ is not periodic then 
$\min_{1\le i <j \le N_0} d(T^i\zeta, T^j \zeta) \ge s(\zeta)>0$
and hence $\mu (U_n\cap T^{-j}U_n)=0$  for all $1\le j \le N_0$ and  $n$ large enough.

Since $u_n$ was chosen so that $n\mu(U_n)\to 1$, we get
\[
\mu (U_n\cap T^{-j}U_n) \le 2n^{-1-\delta}
\]
for any $1\le j\le \log^5n$, and consequently
\[
\lim_{n\rightarrow \infty} n \sum_{j=1}^{\log^5n} \mu (U_n\cap T^{-j}U_n)=0.
\]

\subsubsection{Accounting for exceedances between $\log^5n$ and $\sqrt{n}$.}
 We use  exponential decay of correlations to  show
\begin{equation}\label{limiting.short.terms}
 \lim_{n\rightarrow \infty} n \sum_{j=\log^5n}^{p=\sqrt{n}} \mu (U_n\cap T^{-j}U_n)=0.
\end{equation}
As before, we approximate the indicator function $1_{U_n}$ of  the set $U_n$ 
by a suitable Lipschitz function. 
Recall that $U_n$ is a ball of some radius $r_n \sim \frac{1}{\sqrt{n}}$ centered at the point 
 $\zeta$.  We define $\Phi_n$ to be $1$ inside $B(\zeta,r_n-n^{-\frac{2}{\xi}})$,
  where $\xi$ comes from Assumption A, and decaying to $\Phi_n=0$ on $X\setminus U_n$.
   The Lipschitz norm of $\Phi_n$ is  then bounded by $n^{\frac{2}{\xi}}$. 
Thus
\begin{eqnarray*}
|\int 1_{U_n}(1_{U_n} \circ T^j)\,d\mu - (\int 1_{U_n}~d\mu)^2 | &\le& |\int \Phi_n(\Phi_n \circ T^j)\,d\mu -(\int \Phi_n~d\mu)^2|\\
&+&|(\int \Phi_n~d\mu)^2- (\int 1_{U_n}~d\mu)^2|\\
&+& | \int 1_{U_n}(1_{U_n} \circ T^j)\,d\mu -\int \Phi_n(\Phi_n\circ T^j)\,d\mu |.
\end{eqnarray*}
If  $ (\log n)^5\le j \le p=\sqrt{n}$ then we obtain by decay of correlations for the first term
$$
|\int \Phi_n (\Phi_n \circ T^j)\,d\mu -(\int \Phi_n~d\mu)^2  |\le C n^{\frac{4}{\xi}} \theta^{j}\le \frac{C}{n^2}
$$ 
if  $n$  is sufficiently large. For the second term we obtain for $n$ large enough
$$
|(\int \Phi_n d\mu)^2-(\int 1_{U_n}~d\mu)^2|\le\mu(A_{r_n,n^{-2/\xi}})\le(n^{-2/\xi})^{\xi}< C  n^{-2}.
$$
Similarly we estimate the third term as follows
 $$
 |\int \Phi_n (\Phi_n \circ T^j)\,d\mu - \int 1_{U_n} (1_{U_n}\circ T^j)\,d\mu| 
\le2\mu(A_{r_n,n^{-2/\xi}}) \le\frac{C}{n^2}. 
$$
 Hence equation~\eqref{limiting.short.terms} is satisfies which concludes the proof of Theorem~\ref{Poisson}.

 \section{Billiards with polynomial mixing rates}
 \label{sec:polynomial-rates}

\begin{proof}[Proof of Theorem~\ref{poly}]
 First suppose  $\zeta$ is a generic point in $M$. We may establish a Poisson limit law for nested balls about $\zeta$ by proving $D_3(u_n)$ and $D' (u_n)$ 
 as in the case of Sinai dispersing billiards for  the map $F: M\to M$ with respect to the measure
 $\mu_M$. To prove $D_3(u_n)$  note that local stable manifolds contract exponentially, Assumption A holds (as the measure $\mu_M(\cdot)=\frac{1}{\mu(M)} (\cdot\cap M)$)
 and  the exponential decay  of Equation~\eqref{Lip_infinity_decay} in the Lipschitz norm versus 
 $\mathscr{L}^{\infty}$ holds because we 
 have the structure of a Young Tower for $F:M\to M$. Hence $D_3 (u_n)$ holds for generic points $\zeta$ in $M$. These are the only ingredients of the proof
 for $D_3(u_n)$.
 
 The proof of $D'(u_n)$ also 
 proceeds in the same way as for Sinai dispersing billiards as the local unstable manifolds contract  uniformly under $F^{-1}$, the measure $\mu_M$ decomposes into 
 a conditional measure on the local unstable manifolds which is absolutely continuous with respect to Lebesgue measure. These are the only ingredients of the 
 proof of $D'(u_n)$ for Sinai dispersing billiards.  
 
 Finally we use Proposition~\ref{Extend}   to extend this result to generic points in phase space. 
 \end{proof}

 \begin{proof}[Proof of Proposition~\ref{Extend}]
 The argument below is built on adjustments of the proofs of \cite[Theorem~2.1]{BSTV03} and \cite[Theorem~5]{FFT13}.
Since $N$ is a simple point process, without multiple
  events, we may use a criterion proposed by Kallenberg
  \cite[Theorem~4.7]{K86} to show the stated convergence. Namely we
  need to verify that
\begin{enumerate}
\item $\E(N_n(I))\xrightarrow[]{n\to\infty}\E(N(I))$, for all $I\in
\mathcal S$;

\item $\mu(N_n(J)=0)\xrightarrow[]{n\to\infty}\mu(N(J)=0)$, for all
$J\in\mathcal R$,
\end{enumerate}
where $\E(\cdot)$ denotes the expectation with respect to $\mu$. As before let us put $U_n=\{X_0>u_n\}$.

The first condition follows trivially by definition of the point process $N_n$. In fact, let $a,b\in\R^+$ be such
that $I=[a,b)$, then, recalling that $v_n=1/\mu(U_n)$, we have
\begin{align*}
  \E(N_n(I))&=\E\left(\sum_{j=\lfloor v_na\rfloor+1}^{\lfloor v_nb\rfloor}
\I_{T^{-j}U_n}\right)=\sum_{j=\lfloor v_na\rfloor+1}^{\lfloor
v_nb\rfloor}\E(\I_{T^{-j}U_n})\\
&=\left(\lfloor v_nb\rfloor-(\lfloor v_na\rfloor+1)\right)\mu(U_n)\\
&\sim (b-a)v_n\mu(U_n)\xrightarrow[]{n\to\infty}(b-a)=\E(N(I)).
\end{align*}

To prove (2), note by \cite[Corollary~6]{Z07} we only need to show that 
$$\mu_M(N_n(J)=0)\xrightarrow[]{n\to\infty}\p(N(J)=0), \quad \text{for all
$J\in\mathcal R.$}$$

Let
$$
E_n(x):=\frac1n\sum_{i=0}^{n-1}r_{M}\circ F^i(x)
$$  
then by the ergodic theorem we get for $\mu$-a.e. $x\in M$:
$$
E_n(x)\to c:=\int_{ M}r_{M}~d\mu_M=\frac1{\mu(M)}
$$
where the final equality follows from Kac's Theorem. Moreover $c=v_n/\hat{v}_n$.

For $\mu$-a.e. $x\in M$, there exists a finite number $j(x, \eps)$ such that  $|E_n(x)-c|<\eps$ for all $n\ge j(x, \eps)$.  Let $\tilde G_{n}^\eps:=\{x\in M:j(x,\eps)<n\}$.  Moreover, we define $N=N(\eps)$ to be such that 
\begin{equation}
\label{eq:estimate1}
\hat\mu(\tilde G_N^\eps)>1-\eps.
\end{equation}

Since $$\left|\sum_{i=0}^{n-1}r_{M}(F^i(x))-cn\right|<\eps n \text{ for } x\in \tilde G_N^\eps \text{ and } n\ge N,$$
for all such $n$, there exists $s=s(x)$ with $|s|<\eps n$ such that $F^n(x)=T^{cn+s}(x)$.  
Since $r_{U_n}=\sum_{i=0}^{\hat{r}_{U_n}-1} r_M\circ F^i$, we obtain
$$
r_{U_n}(x)=c\hat r_{U_n}(x)+s
$$
for some $|s|<\eps\hat r_{U_n}(x)$ whenever $\hat r_{U_n}(x)\ge N$ and $x\in \tilde G_N^\eps$,
where we used that $c=v_n/\hat{v}_n$.

Note that since $U_{n+1}\subset U_n\forall n$ the sets $L_{N,n}^\eps:=\{\hat r_{U_n}> N\}$
are nested, i.e. $L_{N,n}^\eps \subset L_{N,n+1}^\eps\forall n$. 
Hence, as  $\mu_M(\hat r_{U_n}\leq j)\leq j\mu_M(U_n)\to 0$, 
as $n\to\infty$ there exists  $N'=N'(\eps)$ sufficiently large such that 
\begin{equation}
\label{eq:estimate2}
\mu_M((L_{N,n}^\eps)^c)<\eps.
\end{equation}
for all $n>N'$.

Let $J_{sup}=\sup J+1$. Observe that
\begin{align*}
\mu_{M}\left(N_n ([0,J_{\sup}))>\kappa\right)
&\leq\mu_{M}\left(\hat N_n (v_n/\hat v_n [0,J_{\sup}))>\kappa\right)\\
&=\mu_{M}\left(\hat N_n( c[0,J_{\sup}))>\kappa\right)\xrightarrow[]{u\to u_F}\p(N([0,cJ_{\sup})>\kappa)\xrightarrow[]{\kappa\to\infty}0.
\end{align*}
This implies that we can choose $K(J)$ independent of $\eps$ such that 
$\mu_{M}\left(N_n ( J)>K(J)\right)<\eps$.

Also, for any $x\in M$ and $i=2,\ldots$, let $r_{U_n}^{(i)}(x):= r_{U(n)}(T^{r_{U_n}^{(i-1)}})(x)$ where
 $r_{U_n}^{(1)}:=r_{U_n}$ and put $\tau^i_{U_n}=\tau^{i-1}_{U_n}+r_{U_n}^{(i)}$, with 
 $\tau^1_{U_n}=r_{U_n}$ for the $i$th return  time to $U_n$ under the map $T$.
 Similarly we define $\hat r_{U_n}^{(i)}(x):= \hat r_{U_n}(F^{\hat r_{U_n}^{(i-1)}})(x)$ 
 and $\hat\tau^i_{U_n}=\hat\tau^{i-1}_{U_n}+\hat{r}_{U_n}^{(i)}$ for the $i$th return time to $U_n$ 
 under $F$.
We will use the ergodic theorem to approximate  $\tau^i_{U_n}(x)$ by $c\hat\tau^i_{U_n}(x)$
on a large set.

For that purpose put
$$
E(u_n,J,\eps):=  \{N_n( J)=0\}\cap\{N_n ([0,J_{\sup}))>K\}\cap 
\left(\bigcap_{j=1}^{K}T^{-\tau^{j}_{U_n}}\left( \tilde G_{N}^{\eps/K}\cap L_{N,N'}^{\eps/K}\right)\right)
$$
By stationarity, \eqref{eq:estimate1} and \eqref{eq:estimate2}, for $K$, $N$  and $n$ sufficiently large  we have 
\begin{multline}
\Big|\mu_{M}(N_n( J)=0)-\mu_{M} \left(E(u_n,J,\eps)\right)\Big|\\ \leq \mu_M(N_n ([0,J_{\sup}))>K)+K\mu_{M}\left( \left(\tilde G_{N}^{\eps/K}\right)^c\right)+K\mu_{M}\left(\left(L_{N,N'}^{\eps/K}\right)^c\right)\leq 3\epsilon.
\label{eq:approx1linha}
\end{multline}
By definition of $\tilde G_{N}^{\eps/K}$  we now conclude that for $x\in E(u_n,J,\eps)$ and
 $j=1,\ldots, K$, there exist $|s_j|<\eps \hat r_{U_n}^{(j)}(x)$ such that
$$
r_{U_n}^{(j)}(x)=c\hat r_{U_n}^{(j)}(x)+s_j.
$$
Hence
\begin{equation}\label{eq:inter-cluster}
\left|\tau_{U_n}^j-c\hat\tau_{U_n}^j\right|\le K\eps
\end{equation}  
on $E(u_n,J,\eps)$ for $j=1,\dots,K$.
Since $\hat v_n=v_n/c$, from \eqref{eq:inter-cluster}, we get that for $x\in E(u_n,J,\eps)$ and every $j=1,\ldots, K$
\begin{equation}
\label{eq:rel1}
\tau_{U_n}^{j}(x)\in v_n J\quad\Rightarrow\quad \hat\tau_{U_n}^{j}(x)\in\hat v_n(1+B(0,K\eps/c))J
\end{equation}
and also
\begin{equation}
\label{eq:rel2}
\hat\tau_{U_n}^{j}(x)\in\hat v_nJ\quad\Rightarrow\quad \tau_{U_n}^{j}(x)\in v_n(1+B(0,K\eps/c))J,
\end{equation}
where we used $(1+B(0,\delta))J=\{x=(1+y)z: |y|<\delta, z\in J\}$.
Hence,
$$
\mu_{M}(\hat N_n(J)=0)\leq\mu_{M} \left(E(u_n,(1+B(0,K\eps/c))J,\eps)\right)
\leq\mu_{M}(\hat N_n ((1+B(0,2K\eps/c))J)=0).
$$
Taking limits as $n\to\infty$, by hypothesis, we get that
 $$\p(N(J)=0)\leq\mu_{M} \left(E(u_n,(1+B(0,K\eps/c))J,\eps)\right)\leq \p(N((1+B(0,2K\eps/c))J)=0).$$ 

Finally, using \eqref{eq:approx1linha} and that $\lim_{\delta\to0}\p(N((1+B(0,\delta))J)=0)=\p(N(J)=0)$
(as $J$ is a finite union of disjoint intervals), we get 
$$
\lim_{n\to\infty}\mu_{M}( N_n (J)=0)=\p(N(J)=0).
$$
\end{proof}

 \section{Some applications of Theorem~\ref{poly} to polynomially mixing billiards.}
 
 Chernov and Zhang~\cite{CZ05} give examples of polynomially mixing billiards to which Theorem~\ref{poly} applies. For example
  to semi-dispersing billiards in rectangles with internal scatters, 
 Bunimovich stadia, Bunimovich flower-like regions and skewed stadia (see the figures above). These billiards have polynomial mixing rates yet
 exhibit Poisson return time statistics.





\begin{figure}[h!]
        \centering
        \begin{subfigure}[b]{0.3\textwidth}
                \includegraphics[width=\textwidth]{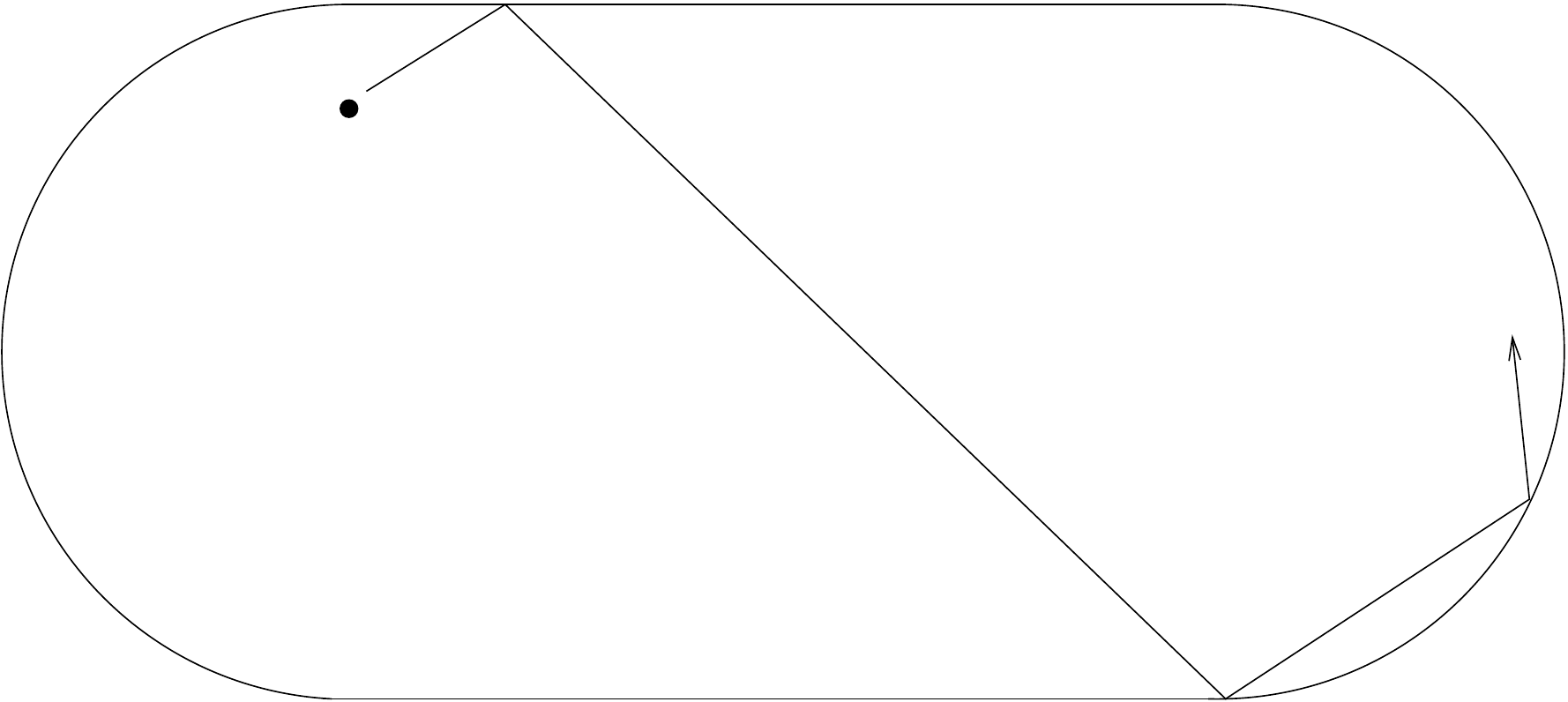}
                \caption{A Bunimovich stadium.}
                \label{fig:bunimovich}
        \end{subfigure}%
        ~ 
        \begin{subfigure}[b]{0.3\textwidth}
                \includegraphics[width=\textwidth]{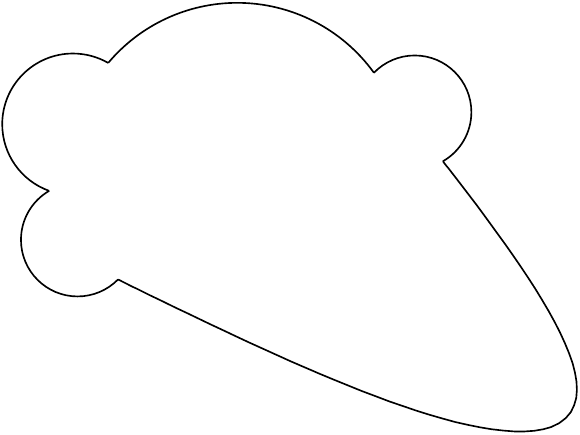}
                \caption{Flower like stadia.}
                \label{fig:flowers}
        \end{subfigure}
        ~ 
                \caption{Some polynomially mixing billiards.}\label{fig:animals}
\end{figure}



 \bibliographystyle{amsalpha}

\bibliography{ExtremeBilliards}

 \end{document}